\documentclass{article}

\usepackage{theorem}
\usepackage{amsmath,amssymb,amsfonts,graphicx,float,bm}
\usepackage{fancyvrb}

\DeclareMathOperator{\sgn}{sgn}

\def\BBox{\rule{2mm}{3mm}}
\def\QED{\hfill$\BBox$}
\newenvironment{proof}
{\begin{rm}\par\smallskip\noindent{\bf Proof.}\quad}{\QED\end{rm}}

\newtheorem{thm}{Theorem}[section]

\newtheorem{prop}[thm]{\bfseries Proposition} 

\theorembodyfont{\rmfamily}

\theorembodyfont{\rmfamily}
\newtheorem{defn}[thm]{\bfseries Definition}

\begin{document}


\title{Enumeration of PLCP-orientations of the $4$-cube} 


\author{
Lorenz Klaus
\\
National Institute of Informatics, Japan\\
and
JST, ERATO, Kawarabayashi \\
Large Graph Project, Japan\\
{\tt lorenz@nii.ac.jp}
\and
Hiroyuki Miyata
\\
Graduate School of Information Sciences, \\
Tohoku University, Japan\\
{\tt hmiyata@dais.is.tohoku.ac.jp}
}

\maketitle

\begin{abstract}
The linear complementarity problem (LCP) provides a unified approach to many problems such as
linear programs, convex quadratic programs, and bimatrix games. The general LCP is known to be NP-hard, but there are some promising results that suggest the possibility that
the LCP with a P-matrix (PLCP) may be polynomial-time solvable. 
However, no polynomial-time algorithm for the PLCP has been found yet and the computational complexity of the PLCP remains open. 
Simple principal pivoting (SPP) algorithms, also known as Bard-type algorithms, are candidates for polynomial-time algorithms for the PLCP.
In 1978, Stickney and Watson interpreted SPP algorithms as a family of algorithms that seek the sink of unique-sink orientations of $n$-cubes. They performed the enumeration of the arising orientations of the $3$-cube, hereafter called PLCP-orientations. In this paper, we present the enumeration of PLCP-orientations of the $4$-cube.
The enumeration is done via construction of oriented matroids generalizing P-matrices and realizability classification of oriented matroids.
Some insights obtained in the computational experiments are presented as well.
\end{abstract}

\section{Introduction}
The {\it linear complementarity problem (LCP)}, introduced by Cottle and Dantzig~\cite{CD68} and Lemke and Howson~\cite{LH64}, is defined as follows.
\begin{align*}
\textbf{LCP($M,{\bm q}$):}&\\
\text{find } &{\bm w},{\bm z} \in \mathbb{R}^n \text{ s.t. }
\\
&{\bm w}-M{\bm z} = {\bm q},\\
&{\bm w},{\bm z} \geq 0, {\bm w}^T{\bm z} = 0
\end{align*}
for given matrix $M \in \mathbb{R}^{n \times n}$ and vector ${\bm q} \in \mathbb{R}^n$.
The LCP has been studied extensively because it provides a unified approach to several important optimization problems
such as linear programs, convex quadratic programs, and bimatrix games. 
The LCP is in general NP-hard~\cite{C89}, but some interesting subclasses can be solved in polynomial time.
For positive semidefinite matrices $M$, for instance, any LCP$(M,q)$ is polynomial-time solvable by
interior point methods~\cite{KMNY90,KMY89}.
The LCP with a {\it K-matrix} of order $n$ can be solved in a linear number in $n$ of pivot steps by \emph{simple principal pivoting algorithms} 
with any pivot rules~\cite{FFGL09}.
LCPs with {\it hidden K-matrices}, which build a superclass of K-matrices, can be solved in polynomial time through solving
linear programs~\cite{M76,PC85}.
Solving LCPs with hidden K-matrices of order $n$ is actually as hard as solving linear programs over combinatorial $n$-cubes
when simple principal pivoting methods are employed~\cite{FFK_pre}.
Among interesting problem classes, the LCP with a {\it P-matrix} (PLCP) has been paid special attention 
since it is a relatively broad class (it includes all the above classes),
admits several natural characterizations, and there are some promising results on its complexity.

\begin{defn}
\label{defn:p_matrix}
A {\it P-matrix} is a matrix $M \in \mathbb{R}^{n \times n}$ whose principal minors are all positive.
\end{defn}  
P-matrices admit several characterizations. 
In the following, let $[n]$ for $n \in \mathbb{N}$ denote the set $\{ 1,\dots,n\}$.
Let $A:=(I_n,-M)$ and 
$A_B$ for $B \subseteq [n]$ be the square submatrix of $A$ whose $i$th column is the $i$th column of $-M$ if $i \in B$,
and the $i$th column of the identity matrix $I_n$ otherwise.
Let $\oplus$ denote the \emph{symmetric difference} operator of sets.
\begin{thm}
\label{thm:P_matrix}
A matrix $M \in \mathbb{R}^{n \times n}$ is a P-matrix if and only if one of the following equivalent conditions holds.
\begin{itemize}
\item[{\rm (A)}] $LCP(M,{\bm q})$ has a unique solution for every ${\bm q} \in \mathbb{R}^n$.
\item[{\rm (B)}]  For any $B \subseteq [n]$ and $i \in [n]$, we have 
     $\det(A_B)\det(A_{B \oplus \{ i\}}) < 0$.
\item[{\rm (C)}] Every nonzero vector ${\bm x} \in \mathbb{R}^n$ satisfies $x_i(M{\bm x})_i > 0$ for some $i \in [n]$.
\end{itemize}          
\end{thm}
Equivalence of (A) and the P-matrix property has been obtained by Samelson, Thrall and Wesler~\cite{STW57}, Ingleton~\cite{I66}, and Murty~\cite{M72} independently.
Characterization (B) is a restatement of Definition \ref{defn:p_matrix} in terms of the matrix $(I_n,-M)$.
Characterization (C) is due to Fiedler and Pt\'ak~\cite{FP62}.
There are many other characterizations of P-matrices. The reader can find good surveys in \cite{FFK11,T84}.
The following result by Megiddo has encouraged researchers to seek a polynomial-time algorithm for the PLCP.
\begin{thm}{\rm (\cite{Me88})}
If the PLCP is NP-hard, then co-NP = NP.
\end{thm}
However, no polynomial-time algorithm for the PLCP is currently known and the computational complexity of the PLCP remains open.

\subsection*{Simple principal pivoting (SPP) algorithms}
Regarding pivoting approaches to the LCP, {\it simple principal pivoting (SPP) algorithms}~\cite{Z60,B74}, {\it Lemke's algorithm}~\cite{L65}, and 
the {\it criss-cross method}~\cite{KT89,FT92} have been proposed as solving methods.
When focusing on the PLCP, Lemke's algorithm and the criss-cross method can also be viewed as SPP algorithms\footnote{This follows from the fact that a simple principal pivot is always possible
for the PLCP at any nonterminal basis, and Lemke's algorithm interpreted
as a parametric LCP algorithm, see \cite[Section 4.5.2]{CPS92}, uses a simple principal pivot whenever possible.}.
SPP algorithms were first introduced by  Zoutendijk~\cite{Z60} and Bard~\cite{B74}, and are also called {\it Bard-type algorithms}.
In the following, we shall explain SPP algorithms.

For a PLCP($M,{\bm q}$), let $A_B$ be as defined above. We denote by $A_B[i,{\bm q}]$ for $i \in [n]$ the $n \times n$ matrix that coincides with $A_B$ except that the $i$th column is replaced by ${\bm q}$. First, we pick any $B \subseteq [n]$. Since $A_B$ is nonsingular, we call $B$ a {\it basis}. Let
\[ w_i := 
\begin{cases}
0 & \text{if $i \in B$,}\\
(A_B^{-1}{\bm q})_i & \text{if $i \notin B$}
\end{cases}
\
\
\text{ and }
\
\
 z_i :=
\begin{cases}
(A_B^{-1}{\bm q})_i & \text{if $i \in B$,}\\
0 & \text{if $i \notin B$}
\end{cases}
\]
for all $i \in [n]$.
If $A_B^{-1}{\bm q} \geq 0$, then ${\bm z}:=(z_1,\dots,z_n)^T$ and ${\bm w}:=(w_1,\dots,w_n)^T$ is the unique solution.
If $A_B^{-1}{\bm q} \not\geq 0$, then ${\bm w}-M{\bm z} = {\bm q}$ and ${\bm w}^T{\bm z}=0$ are satisfied, but
${\bm w},{\bm z} \geq 0$ fails.
To improve the situation, we choose an index $i \in [n]$ with $(A_B^{-1}{\bm q})_i < 0$.
We have $(A_B^{-1}{\bm q})_i = \det(A_B)^{-1}\det(A_B[i,{\bm q}])$ by Cramer's rule.
Since $M$ is a P-matrix, it holds that
\begin{align*}
(A_B^{-1}{\bm q})_i (A_{B \oplus \{ i \}}^{-1}{\bm q})_i & = \det(A_B)^{-1}\det(A_B[i,{\bm q}])\det(A_{B \oplus \{ i \}})^{-1}\det(A_{B \oplus \{ i \}}[i,{\bm q}]) \\
& = \det(A_B)^{-1}\det(A_{B \oplus \{ i \}})^{-1}\det(A_B[i,{\bm q}])^{2} < 0
\end{align*}
by Characterization (B) in Theorem \ref{thm:P_matrix}.
Now, we set $B:= B \oplus \{ i \}$ and obtain $(A_B^{-1}{\bm q})_i > 0$. The signs at other indices can change arbitrarily.
Such a basis exchange operation is a {\it pivot}.

SPP algorithms keep pivoting until a solution is found. We obtain different algorithms by the choice of a {\it pivot rule}, which uniquely selects an index $i$ with $(A_B^{-1}{\bm q})_i < 0$ in each iteration. Some rules can cause cycles, and thus a pivot rule has to be chosen carefully so that the algorithm will terminate in a finite number of iterations. The first finite pivot rule has been proposed by Murty~\cite{M74}.
Till this day, neither a deterministic nor randomized (simple) pivot rule is known that is guaranteed to terminate in a polynomial number in $n$ of pivot steps. Such a rule would yield a strongly polynomial-time solving method for the PLCP.

\subsection*{PLCP-orientations}
Stickney and Watson~\cite{SW78} introduced a digraph model for SPP algorithms.

Consider a PLCP$(M,{\bm q})$ of order $n$. We assume \emph{nondegeneracy}. The instance is \emph{nondegenerate} if $\det(A_B[i,{\bm q}]) \neq 0$ for every $B \subseteq [n]$ and $i \in [n]$. Let the \emph{$n$-cube} graph be the graph with vertex set $\{ B \, | \, B \subseteq [n]\}$ and edge set 
$$\{ \{B,C\} \, | \, B,C \subseteq [n] \text{ and } |B \oplus C|=1 \}.$$
In other words, two bases are adjacent if and only if they differ in exactly one entry. We assign an orientation to each edge. The edge connecting $B \subseteq [n]$ with an adjacent $B \oplus \{i\}$ for $i \in [n]$ is oriented from $B$ towards $B \oplus \{i\}$ if $(A_{B}^{-1}{\bm q})_i < 0$, or equivalently if
\begin{align} 
\label{PLCP}
\det(A_{B})\det(A_{B}[i,{\bm q}]) < 0.
\end{align}
The orientation of the edge is in correspondence with the orientation computed from the viewpoint of $B \oplus \{i\}$. There is a unique \emph{sink}, which is the vertex $S \subseteq [n]$ with $A_{S}^{-1}{\bm q} \geq 0$ corresponding to the unique solution to the PLCP. 
In this model, SPP algorithms behave as path-following algorithms that seek the sink. The orientations of $n$-cubes arising in this way from PLCPs have some good combinatorial properties: 
the {\it unique-sink orientation property}~\cite{SW78} and the {\it Holt-Klee property}~\cite{GMR08}. We call them {\it PLCP-orientations} in the remainder. They are also known as \emph{P-matrix unique-sink orientation} (P-USOs) in the literature.


\subsection*{Unique-sink orientations (USOs)}
An orientation of a polytopal graph is a {\it unique-sink orientation (USO)}
if it has a unique source and unique sink in every nonempty face.
USOs provide a combinatorial model for many optimization problems, such as {\it linear programming} (LP) and the \emph{smallest enclosing ball} (SEB) problem~\cite{SW01}. An \emph{LP-orientation} is a USO of a polytopal graph arising from a linear objective function. USOs of $n$-cubes are particularly interesting. They not just model the PLCP, but also the SEB problem. Even general LP can be reduced to finding the sink of a USO of the $n$-cube~\cite{GS06}. Szab\'o and Welzl~\cite{SW01} proved that finding the sink of a general USO of the $n$-cube can be done by evaluating $O(1.61^n)$ vertices.

PLCP-orientations build a proper subclass of the USOs of the $n$-cube, which follows from enumeration results~\cite{SW78} and counting results~\cite{FGKS13} on USO classes. It is unknown whether there exists a combinatorial characterization of PLCP-orientations.

\subsection*{Our contribution}
SPP algorithms have been studied extensively as candidates for polynomial-time algorithms for the PLCP.
However, no efficient pivot rule is known,
and thus further investigation of PLCP-orientations is of importance.
It would therefore be useful to have a database of PLCP-orientations available.
In this paper, we study the enumeration of PLCP-orientations.
\begin{quote}
{\bf Problem:} Enumerate all PLCP-orientations of the $n$-cube.
\end{quote}

We proceed in two steps.
Since no combinatorial characterization of PLCP-orientations is known, 
we first enumerate a suitable superset.
Todd~\cite{T84} introduced a combinatorial abstraction of the LCP in terms of {\it oriented matroids}, which is known as 
the {\it oriented matroid complementarity problem} (OMCP).
Oriented matroids are characterized by simple axioms, and efficient enumeration algorithms are available~\cite{BGdO00,FF02,FF03}.
The PLCP-orientations are combinatorial objects, and thus extend to the setting of oriented matroids. We then speak of {\it POMCP-orientations}. 

After the enumeration of POMCP-orientations is completed, we check realizability of the associated oriented matroids in order to
identify the PLCP-orientations.
In general, checking oriented matroids for realizability is as difficult as solving general polynomial equality and inequality systems
with integer coefficients~\cite{M88}.
We use the method that has been recently proposed by Fukuda, Miyata, and Moriyama~\cite{FMM12}, which decides realizability
in reasonable time for small instances. 
In this way, enumeration of PLCP-orientations of the $4$-cube is performed.
\begin{thm}
There are $6,910$ POMCP-orientations of the $4$-cube up to isomorphism\footnote{The PLCP-orientations are closed under isomorphism~\cite{SW78}.}, and all of them are also PLCP-orientations.
\end{thm}
The related data is available at
{\tt https://sites.google.com/site/hmiyata1984/lcp}.

Through the computational experiments, we also gained new insights into the combinatorial structure of P-matrices. Furthermore, we came across an interesting subclass of P-matrices whose characterization is completely combinatorial. The related details are discussed in Section \ref{sec:exp}.

\subsection*{Related work}
The first study of PLCP-orientations is due to Stickney and Watson~\cite{SW78}, who
enumerated the PLCP-orientations and the USOs of the $3$-cube.
G\"artner and Kaibel~\cite{GK98} enumerated the LP-orientations of the $3$-cube.
Morris~\cite{M02} proved that every LP-orientation of the $n$-cube is a PLCP-orientation.
Schurr~\cite{S04} enumerated the USOs of the $4$-cube. The results are summarized in Table~\ref{tab:1}.

\begin{table} [h!]
 \centering
 \begin{tabular}{lcc}
   & $3$-cube & $4$-cube \\ [0.5ex] \hline \\ [-2ex]
  USOs & 19  (\cite{SW78}) & 14,614 (\cite{S04}) \\
  PLCP-orientations & 17 (\cite{SW78}) & \phantom{1}6,910 \phantom{([42])} \\
	LP-orientations & 16 (\cite{GK98}) & unknown
 \end{tabular}
 \caption{The number of USOs up to isomorphism} \label{tab:1}
\end{table}

\section{Preliminaries on oriented matroids}
Oriented matroids are a useful tool for analyzing pivoting methods for LP and the LCP.
Here, we summarize the basics of oriented matroids. For a complete overview on oriented matroids, the reader is referred to~\cite{BLSWZ99}.

Before presenting axiomatic systems for oriented matroids, we introduce some operations on sign vectors.
Let $E$ be a finite set, hereinafter referred to as the \emph{ground set}.
For $X,Y \in \{ +,-,0\}^E$, the {\it composition} $X \circ Y \in \{ +,-,0 \}^E$ is given by
\[(X \circ Y)_e := 
\begin{cases}
X_e & \text{if $X_e \neq 0$,} \\
Y_e & \text{otherwise.}
\end{cases}
\]
The {\it separation set} of $X$ and $Y$ is the set
\[ S(X,Y):= \{ e \in E \mid X_e = - Y_e \neq 0 \}.\]
Multiplication of signs is analogous to the multiplication of elements in $\{ 1,-1,0 \}$. Sign vector $X$ \emph{conforms to} $Y$, denoted by $X \preceq Y$, if
$X_e=Y_e$ or $X_e=0$ for all $e \in E$.
\begin{defn}(Vector axioms)\\  
Let $E$ be the ground set. An {\it oriented matroid} in vector representation is a pair ${\cal M}=(E,{\cal V})$, 
where ${\cal V} \subseteq \{ +,-,0\}^E$, that satisfies the following axioms.
\begin{itemize}
\item[(V1)] 
${\bm 0}:=(0,0,\dots, 0) \in {\cal V}$.
\item[(V2)] 
$X \in {\cal V}$ implies $-X \in {\cal V}$.
\item[(V3)] 
For any $X,Y \in {\cal V}$, we have $X \circ Y \in {\cal V}$.
\item[(V4)] 
For any $X,Y \in {\cal V}$ and $e \in S(X,Y)$, \\
there exists $Z \in {\cal V}$ such that $Z_e=0$ and $Z_f=(X \circ Y)_f = (Y \circ X)_f$ 
for all $f \notin S(X,Y)$.
\end{itemize}
An element of the set ${\cal V}$ is a {\it vector} of ${\cal M}$.
\label{vector_axioms}
\end{defn}
Suppose that $|E|=n$. The {\it rank} of an oriented matroid $(E,{\cal V})$ is $r$, where $n-r$ is the length of a maximal chain of vectors:
${\bm 0} \prec X_1 \prec \dots \prec X_{n-r}$ for $X_1,\dots,X_{n-r} \in {\cal V}$.
Let OM($r,n$) denote the set of all rank $r$ oriented matroids with $n$ elements.

Oriented matroids can be specified by alternative axiomatic systems. They, for instance, admit a characterization in terms of basis orientations.
\begin{defn}(Chirotope axioms) \label{def:ChiroAxioms}\\
Let $E$ be the ground set and $r \in \mathbb{N}$ such that $r \leq |E|$. A {\it chirotope} of rank $r$ is a  map  $\chi : E^{r} \rightarrow \{ +,-,0\}$ 
that satisfies the following properties.
\begin{itemize}
\item[(C1)] $\chi$ is not identically zero.
\item[(C2)] $\chi (i_{\sigma (1)},\dots,i_{\sigma (r)}) = {\rm sgn} (\sigma) \chi (i_{1},\dots,i_{r})$
for all $i_1,\dots,i_r \in E$ and any permutation $\sigma$ on $[r]$.
\item[(C3)] For all $i_{1},\dots,i_{r},j_{1},\dots,j_{r} \in E$, the following holds.
\begin{equation*}
\begin{split}
\text{If } \chi (j_{s},i_{2},\dots,i_{r})\cdot \chi (j_{1},\dots,j_{s-1},i_{1},j_{s+1},\dots,j_{r}) \geq 0 \text{ for all } & s \in [r], \text{ then}\\
& \chi (i_{1},\dots,i_{r}) \cdot \chi (j_{1},\dots,j_{r}) \geq 0.
\end{split}
\end{equation*}
\end{itemize}
\label{chirotope_axioms}
\end{defn} 
We remark that (C3) can be replaced by the following axiom.
\begin{itemize}
\item[(C3')] For all $i_{1},\dots,i_{r},j_{1},\dots,j_{r} \in E$, the set 
$$\{ \chi (j_{s},i_{2},\dots,i_{r})\cdot \chi (j_{1},\dots,j_{s-1},i_{1},j_{s+1},\dots,j_{r}) \mid s \in [r] \} \cup \{ - \chi (i_{1},\dots,i_{r}) \cdot \chi (j_{1},\dots,j_{r}) \}$$
is a superset of $\{ +,- \}$ or equals $\{ 0 \}$.
\end{itemize}

Every rank $r$ oriented matroid ${\cal M}=(E, {\cal V})$ is also represented by a pair $(E, \{ \chi, -\chi \})$ for some chirotope $\chi$ of rank $r$. 
The chirotope and vectors can be obtained from each other. See, for instance, Finschi's thesis~\cite{F}. 

Consider a vector configuration $V=({\bm v_i})_{i \in [n]} \in \mathbb{R}^{r \times n}$. The associated map $\chi_V: [n]^r \to \{+,-,0\}$ given by
\[ (i_1,\dots,i_r) \mapsto \sgn \det({\bm v_{i_1}},\dots,{\bm v_{i_r}}) \]
is a chirotope of rank $r$ and specifies an oriented matroid ${\cal M}:=([n],\{ \chi_V, -\chi_V\})$. 
The vectors $\cal V$ of $\cal M$ are then given by
$${\cal V}:=\{(\sgn x_1, \ldots, \sgn x_n) \mid V {\bm x} = 0 \}.$$
Many oriented matroid cannot be represented in terms of any vector configuration. 
A rank $r$ oriented matroid ${\cal M}=([n], \{ \chi, -\chi\})$ is {\it realizable} if there exists a vector configuration 
$V \in \mathbb{R}^{r \times n}$ such that $\chi = \chi_V$, and {\it nonrealizable} otherwise. 
In the former case, matrix $V$ is a \emph{realization} of $\cal M$.

For an oriented matroid ${\cal M}=(E,\{ \chi, -\chi \})$ of rank $r$, a set $\{i_1, \ldots, i_r \} \subseteq E$ with $\chi(i_1, \ldots, i_r) \neq 0$ 
is a \emph{basis}. Then $\cal M$ is {\it uniform} if $\chi (i_1,\dots,i_r) \neq 0$ for any distinct $i_1,\dots,i_r \in E$. 

For $F \subseteq E$, the pair $(F, \{ \chi|_F, -\chi|_F \})$ is an oriented matroid. 
It is the {\it restriction} of ${\cal M}$ by $F$ and denoted by ${\cal M}|_F$. 

A {\it single-element extension} of ${\cal M}$ is a rank $r$ oriented matroid ${\widehat {\cal M}}$ on ground set 
$E \cup \{ p\}$ for $p \notin E$ that satisfies ${\widehat {\cal M}}|_{E} = {\cal M}$. 
In the remainder, whenever speaking of an extension, we refer to a single-element extension.

There are several natural concepts of isomorphism for oriented matroids.
\begin{defn} Let ${\cal M}=(E,\{ \chi, -\chi \})$ be a rank $r$ oriented matroid. \label{prop:ReEquiv}
\begin{itemize}
\item For any permutation $\sigma$ on $E$, the pair $(E,\{ \sigma \cdot \chi, -\sigma \cdot \chi \})$ 
with $\sigma \cdot \chi : E^{r} \rightarrow \{ +,-,0\}$ defined by
$$(i_{1},\dots,i_{r}) \mapsto \chi (\sigma (i_{1}), \dots, \sigma (i_{r}))$$
is an oriented matroid. It is denoted by $\sigma \cdot {\cal M}$.

\item For any $A \subseteq E$, the pair $(E,\{{_{-A} \chi},-{_{-A} \chi} \})$ with ${_{-A} \chi} : E^{r} \rightarrow \{ +,-,0\}$ defined by
$$(i_1,\dots,i_r) \mapsto (-1)^{|A \cap \{ i_1,\dots,i_r\} |}\chi(i_1,\dots,i_r)$$
is an oriented matroid. It is denoted by $_{-A} {\cal M}$
\end{itemize}
\end{defn}

Two oriented matroids $\cal M$ and $\cal N$ on $E$ are \emph{relabeling equivalent} 
if $N=\sigma \cdot {\cal M}$ for some permutation $\sigma$ on $E$. 
They are  $\emph{reorientation equivalent}$ if $\cal M$ and $_{-A} \cal N$ are relabeling equivalent 
for some $A \subseteq E$. These notions motivate the definition of a {\it relabeling class} and 
{\it reorientation class} of oriented matroids. Both operations preserve realizability.


\subsection{The oriented matroid complementarity problem (OMCP)}
For a matrix $M \in \mathbb{R}^{n \times n}$ and a vector ${\bm q} \in \mathbb{R}^n$, the collection
\begin{align*} 
{\cal V}(M,{\bm q}) := 
\{ (\sgn x_1,\ldots,\sgn x_{2n+1}) \in \{ +,-,0\}^{2n+1} \mid 
 (I_n,-M,-{\bm q}){\bm x} = 0, {\bm x} \in \mathbb{R}^{2n+1} \}
\end{align*}
of sign vectors satisfies the vector axioms of an oriented matroid. Once a sign vector $X \in {\cal V}(M,{\bm q})$ with $X \geq 0$, $X_{2n+1}=+$, and $X_i = 0$ or $X_{i+n}=0$ for each $i \in [n]$ is known, a solution to the LCP$(M,{\bm q})$ is obtained in polynomial time. This motivates the generalization of the LCP in the setting of oriented matroids, which is due to Todd~\cite{T84}.

For $n \in \mathbb{N}$, we consider ground sets $[2n]$ and $[2n +1]$, where the pairs $(i,i+n)$ for $i \in [n]$ are the \emph{complementary elements}. The \emph{complement} of any $i \in [2n]$ is denoted by $\overline{i}$. For $A \subseteq [2n]$, let $\overline{A}:=\{\overline{i} \mid i \in A\}$ . 

For an oriented matroid $\widehat{\cal M}=([2n+1],\widehat{\cal V})$, the {\it oriented matroid complementarity problem} (OMCP) is defined as follows. 
\begin{align*}
\textbf{OMCP(${\widehat {\cal M}}$):}&\\
\text{find } & \text{$X \in {\widehat {\cal V}}$ such that}\\
&X_i \cdot X_{\bar{i}} = 0  \text{ for all $i \in [n]$,}\\
&X \geq 0, X_{2n+1} = +.
\end{align*}

Many solving methods for the LCP translate into the setting of oriented matroids. Todd~\cite{T84} generalized Lemke's algorithm~\cite{L65}.
Klafszky and Terlaky~\cite{KT89} and Fukuda and Terlaky~\cite{FT92} proposed the criss-cross method.
Foniok, Fukuda, and Klaus~\cite{FFK11} proved that every SPP algorithm terminates in linear number of pivot steps on {\it K-matroid OMCPs}, 
which generalizes an algebraic result in~\cite{FFGL09}.

Todd~\cite{T84} considered an abstraction of P-matrices and gave several alternative characterizations.
\begin{defn}(\cite{T84})
\label{defn:p_matroid}
A {\it P-matroid} is a rank $n$ oriented matroid ${\cal M}=([2n],{\cal V})$ satisfying one of the following equivalent conditions.
\begin{itemize}
\item[(A)]
For every extension $\widehat{{\cal M}}$ of ${\cal M}$, the OMCP($\widehat{{\cal M}}$) has a unique solution.
\item[(B)] 
Let $\chi$ be a chirotope of ${\cal M}$. Then, we have
\begin{align*} 
\chi(b_1,\dots,b_{i-1},i,b_{i+1},\dots,b_n) = -\chi(b_1,\dots,b_{i-1},\bar{i},b_{i+1},\dots,b_n) \neq 0 
\end{align*}
for all $(b_1,\dots,b_n) \in \{ 1,\bar{1}\} \times \dots \times \{ n, \bar{n}\}$ and $i \in [n]$.
\item[(C)]
For every $X \in {\cal V}$, we have $X_i = X_{\bar{i}} \neq 0$ for some $i \in [n]$.
\end{itemize}
\end{defn}
Other equivalent definitions can be found in~\cite{FFK11,T84}. 

For every P-matrix $M \in \mathbb{R}^{n \times n}$, matrix $(I_n,-M)$ realizes a P-matroid. Conversely, for every realization $(A_B,-A_N) \in \mathbb{R}^{n \times 2n}$ of a P-matroid, matrix $A_B \in \mathbb{R}^{n \times n}$ is nonsingular and $A_B^{-1}A_N$ is a P-matrix.

A {\it POMCP} is an OMCP on an extension ${\widehat {\cal M}}=([2n+1],\{ \widehat{\chi}, -\widehat{\chi}\})$ of a P-matroid.
The problem is {\it nondegenerate} if ${\widehat \chi}(b_1,\dots,b_{i-1},2n+1,b_{i+1},\dots,b_n) \neq 0$
for all $(b_1,\dots,b_n) \in \{ 1,\bar{1}\} \times \dots \times \{ n,\bar{n}\}$ and $i \in [n]$.

Since SPP algorithms make combinatorial decisions, they immediately translate into the oriented matroid setting. The model of an $n$-cube digraph applies again. Consider any nondegenerate POMCP$(\widehat M)$. For a vertex $B \subseteq [n]$ of the $n$-cube, the edge connecting $B$ with an adjacent $B \oplus \{i\}$ for $i \in [n]$ is oriented towards $B \oplus \{i\}$ if
\begin{align}
\label{OMCP}
 {\widehat \chi}(b_1,\dots,b_n)\cdot {\widehat \chi}(b_1,\dots,b_{i-1},2n+1,b_{i+1},\dots,b_{n}) = -,
\end{align}
where $b_j = j+n$ for $j \in B$ and $b_j = j$ otherwise for all $j \in [n]$. 
Compare \eqref{OMCP} with \eqref{PLCP}.
Vertex $B$ is a sink if and only if
\[ {\widehat \chi}(b_1,\dots,b_n)={\widehat \chi}(2n+1,b_2,b_3,\dots,b_n)={\widehat \chi}(b_1,2n+1,b_2,\dots,b_n)=\dots={\widehat \chi}(b_1,\dots,b_{n-1},2n+1),\]
which is equivalent to the existence of the solution vector $X$ to the POMCP$(\widehat M)$ with
\[ X_i=
\begin{cases}
+ & \text{for $i \in \{ b_1,\dots,b_n,2n+1\}$,}\\
0 & \text{otherwise.}
\end{cases}\]

The arising orientations of the $n$-cube satisfy the USO property~\cite{K12}. We call them {\it POMCP-orientations}. Every PLCP-orientation is obviously a POMCP-orientation.

\begin{prop} \label{prop:realPOMCP}
Every POMCP-orientation arising from a realizable oriented matroid is a PLCP-orientation.
\end{prop}
\begin{proof}
Consider any realizable rank $n$ oriented matroid ${\widehat {\cal M}}$ that defines a POMCP-orientation. Let $(A_B,-A_N,-{\bm q})$ for $A_B, A_N \in \mathbb{R}^{n \times n}$ and ${\bm q} \in \mathbb{R}^n$ be a realization of ${\widehat {\cal M}}$. Since ${\widehat {\cal M}}|_{[2n]}$ is a P-matroid, we have ${\widehat \chi}(1,\dots,n) \neq 0$ 
and thus $\det (A_B) \neq 0$.
Then $(I_n,-A_B^{-1}A_N,-A_B^{-1}{\bm q})$ is another realization of ${\widehat {\cal M}}$. The matrix $A_B^{-1}A_N$ is a P-matrix. The LCP($A_B^{-1}A_N,A_B^{-1}{\bm q}$) induces the same orientation as ${\widehat {\cal M}}$.
\end{proof}

\section{Enumeration of PLCP-orientations}
Since no combinatorial characterization of PLCP-orientations is known, we first begin with an enumeration of POMCP-orientations.
The following is an overview of our algorithm.
\begin{enumerate}
\item Enumerate all uniform P-matroids in OM($n,2n$).
\item Partition the set of uniform P-matroids in OM($n,2n$) with respect to a certain equivalence relation and store representatives.
\item For each equivalence class, compute all uniform extensions of the representative oriented matroid of the class.
\item For each extension, compute the induced POMCP-orientations and partition the set of extensions
with respect to the induced orientations.
\item For each class of extensions, decide whether there is a realizable oriented matroid in the class.
\end{enumerate}
There are exactly $n!2^n$ isomorphisms of the $n$-cube
and thus isomorphism test of its orientations is relatively easy for small $n$.
Hence, we will not pay much attention to this issue.

\subsubsection*{{\bf Step 1.}}
The enumeration of P-matroids can be viewed as a variant of the matroid polytope completion problem, for which a software package based on a backtracking method is available~\cite{B}.
In this paper, however, we focus on enumerations of the PLCP-orientations of the $3$-cube and the $4$-cube, and thus it suffices to consider 
P-matroids in OM($3,6$) and OM($4,8$), which are covered by the database of oriented matroids by Finschi and Fukuda~\cite{FF,FF02,FF03}.
The database consists of representatives for the reorientation classes. In every class, the oriented matroid with maximal reverse lexicographic expression of its chirotope is selected as the representative  (see Table~\ref{database_rep}). For a given representative, the oriented matroids in the same reorientation class can be generated easily.
For more information on the database, see \cite{FF,FF03}.

\vspace{+5mm}

\begin{table}[h]
\begin{center}
\begin{BVerbatim}
            1111211121121231112112123112123123411121121231121231234112123123412345
            2223322332334442233233444233444555522332334442334445555233444555566666
            3344434445555553444555555666666666634445555556666666666777777777777777
            4555566666666667777777777777777777788888888888888888888888888888888888
IC(8,4,1) = ++++++++++++++++++++++++-++++------++++++----+------------------------
IC(8,4,2) = +++++++++++++++++++++++++++-------++------------+--++-----+--+---+--++
IC(8,4,3) = +++++++++++++++++++++++++++-------++--+--+-++--++++++++--+++++++++++--
IC(8,4,4) = +++++++++++++++++++++++++++-------++++++++--+++-------++++--+---++++-+
\end{BVerbatim}
\end{center}
\caption{Chirotope representations in the database of Finschi and Fukuda~\cite{FF}}
\label{database_rep}
\end{table}

There exist nonuniform P-matroids, but every PLCP-orientation arises also from a uniform P-matroid.
Therefore, it is enough to extract the uniform P-matroids only.
\begin{prop}
For every extension ${\widehat {\cal M}}$ of a P-matroid that defines a nondegenerate POMCP, 
there exists a uniform extension of a (uniform) P-matroid that induces the same POMCP-orientation as ${\widehat {\cal M}}$.
\end{prop}
\begin{proof}
We apply a perturbation argument. Suppose that $\widehat{\cal M}=([2n+1], \{ \widehat{\chi}, -\widehat{\chi} \})$ is of rank $n$. 
Let $\epsilon$ be any uniform chirotope of rank $n$ on $[2n+1]$. Consider the map $\widehat{\chi}_{\epsilon}:[2n+1]^n \rightarrow \{ +,-,0\}$ defined by
\[ (i_1, \ldots, i_n) \mapsto {\widehat \chi} (i_1, \ldots, i_n) \circ \epsilon(i_1, \ldots, i_n)\]

The map $\widehat{\chi}_{\epsilon}$ satisfies the chirotope axioms.
Since it is clear that it satisfies (C1) and (C2), we only verify (C3).
Suppose that there exist $i_1,\dots,i_n$, $j_1,\dots,j_n \in [2n+1]$ such that
\[ \widehat{\chi}_{\epsilon}(j_s,i_2,\dots,i_n) \cdot \widehat{\chi}_{\epsilon}(j_1,\dots,j_{s-1},i_1,j_{s+1},\dots,j_n) \geq 0 \] 
for all $s \in [n]$ while $\widehat{\chi}_{\epsilon}(i_1,\dots,i_n) \cdot \widehat{\chi}_{\epsilon}(j_1,\dots,j_n) = -$. 
By the definition of $\widehat{\chi}_{\epsilon}$, it holds that 
\[ {\widehat \chi}(j_s,i_2,\dots,i_n) \cdot {\widehat \chi}(j_1,\dots,j_{s-1},i_1,j_{s+1},\dots,j_n) \geq 0 \]
for all $s \in [n]$. Since ${\widehat \chi}$ is a chirotope, this leads to ${\widehat \chi}(i_1,\dots,i_n)\cdot {\widehat \chi}(j_1,\dots,j_n) \geq 0$.
First, suppose that ${\widehat \chi}(i_1,\dots,i_n)\cdot {\widehat \chi}(j_1,\dots,j_n) = +$. 
Then we have $\widehat{\chi}_{\epsilon}(i_1,\dots,i_n)\cdot \widehat{\chi}_{\epsilon}(j_1,\dots,j_n) = +$, which is a contradiction. 
Secondly, suppose that ${\widehat \chi}(i_1,\dots,i_n)\cdot {\widehat \chi}(j_1,\dots,j_n) = 0$.
By using (C3'), we obtain 
\[ {\widehat \chi}(j_s,i_2,\dots,i_n) \cdot {\widehat \chi}(j_1,\dots,j_{s-1},i_1,j_{s+1},\dots,j_n) = 0 \]
for all $s \in [n]$. Therefore, we have
\begin{align*}
 0 & \leq \widehat{\chi}_{\epsilon}(j_s,i_2,\dots,i_n) \cdot \widehat{\chi}_{\epsilon}(j_1,\dots,j_{s-1},i_1,j_{s+1},\dots,j_n) \\ 
& = \epsilon(j_s,i_2,\dots,i_n) \cdot \epsilon(j_1,\dots,j_{s-1},i_1,j_{s+1},\dots,j_n) 
\end{align*}
for all $s \in [n]$.
Since $\epsilon$ is also a chirotope, we have 
$\epsilon(i_1,\dots,i_n) \cdot \epsilon(i_1,\dots,i_n) \geq 0$ and thus $\widehat{\chi}_{\epsilon}(i_1,\dots,i_n) \cdot {\widehat \chi}_{\epsilon}(i_1,\dots,i_n) \geq 0$, which is a contradiction.

The oriented matroid $\widehat {\cal M}_{\epsilon}:=([2n+1],\{ \widehat{\chi}_{\epsilon}, -\widehat{\chi}_{\epsilon}\} )$ is uniform, induces the same orientation as $\widehat {\cal M}$, and the restriction $\widehat {\cal M}_{\epsilon}|_{[2n]}$ is a uniform P-matroid.
\end{proof}
\\
\\
Note that if $\widehat{\cal M}$ is realizable, then there always exists a perturbed oriented matroid $\widehat{\cal M}_{\epsilon}$ that is realizable. 

\subsubsection*{{\bf Step 2.}}
To reduce the number of P-matroids in the computation, we partition the set of P-matroids into certain equivalence classes.
\begin{defn}
Let ${\cal M}$ and ${\cal N}$ be rank $n$ oriented matroids on ground set $[2n]$.
\begin{itemize}
\item 
${\cal M}$ and ${\cal N}$ are {\it C-equivalent} 
if ${\cal N} = \sigma \cdot {\cal M}$ for some permutation $\sigma$ on $[2n]$ with $\sigma (\overline{i}) = \overline{\sigma (i)}$ for all $i \in [2n]$.

\item
${\cal M}$ and ${\cal N}$ are {\it FS-equivalent} if ${\cal N} = {_{-A}{\cal M}}$ for some $A \subseteq [2n]$ with
$\overline{A}=A$.

\item ${\cal M}$ and ${\cal N}$ are {\it CFS-equivalent} if they can be transformed into each other by a combination of the above two operations.
\end{itemize}
\end{defn}
The concept of C-equivalence is related to isomorphisms of USOs. The definition of FS-equivalence is motivated by the notion of face switches for USOs introduced by Stickney and Watson~\cite{SW78}. In this paper, however, we use the terminology facet switch instead.
\begin{defn}
 For a USO of the $n$-cube, a {\it facet switch} is the operation of reversing the orientation of all edges connecting any two opposite facets.
\label{defn:facet_switch}
\end{defn}
PLCP-orientations are closed under isomorphism and facet switches~\cite{SW78}. This convenient property translates into the setting of oriented matroids. It follows as a corollary of Proposition~\ref{prop:IsoFacPUSO} below.
\begin{prop}
Let ${\cal M}$ be a P-matroid. Every oriented matroid $\cal N$ that is C- or FS-equivalent to ${\cal M}$ is a P-matroid.
\end{prop}
\begin{proof}
First, suppose that $\cal N$ is C-equivalent to ${\cal M}=([2n+1],\{ \chi,-\chi \} )$ of rank $n$. 
Let $\sigma$ be a permutation on $[2n]$ with $\sigma (\overline{i}) = \overline{\sigma (i)}$ for all $i \in [2n]$
such that ${\cal N} = \sigma \cdot {\cal M}$.
Pick an arbitrary $(b_1,\dots,b_n) \in \{ 1,\bar{1}\} \times \dots \{ n, \bar{n}\}$.
By (B) in Definition \ref{defn:p_matroid}, we have
\begin{align*}
\sigma \cdot \chi(b_1,\dots,b_{i-1},i,b_{i+1},\dots,b_n) 
&=  \chi (\sigma (b_1),\dots,\sigma (b_{i-1}),\sigma(i),\sigma (b_{i+1}),\dots,\sigma (b_n)) \\
&= -\chi (\sigma (b_1),\dots,\sigma (b_{i-1}),\overline{\sigma(i)},\sigma (b_{i+1}),\dots,\sigma (b_n)) \\
&= -\chi (\sigma (b_1),\dots,\sigma (b_{i-1}),\sigma(\overline{i}),\sigma (b_{i+1}),\dots,\sigma (b_n)) \\
&= - \sigma \cdot \chi (b_1,\dots,b_{i-1},\bar{i},b_{i+1},\dots,b_n).
\end{align*}
for every $i \in [n]$.
Hence ${\cal N}=([2n+1],\{ \sigma \cdot \chi, -\sigma \cdot \chi \} )$ is a P-matroid.

Secondly, suppose that $\cal N$ is FS-equivalent to $\cal M$.
Pick $A \subseteq [2n]$ with $\overline{A} = A$ such that ${\cal N}={_{-A}{\cal M}}$.
Then, we have
\begin{align*}
{_{-A}\chi(b_1,\dots,b_{i-1},i,b_{i+1},\dots,b_n)} 
&= (-1)^{|A \cap \{ b_1,\dots,b_{i-1},i,b_{i+1},\dots,b_n \}|}\chi(b_1,\dots,b_{i-1},i,b_{i+1},\dots,b_n) \\
&= (-1)^{|A \cap \{ b_1,\dots,b_{i-1},\bar{i},b_{i+1},\dots,b_n \}|}\chi(b_1,\dots,b_{i-1},i,b_{i+1},\dots,b_n) \\
&= (-1)\cdot(-1)^{|A \cap \{ b_1,\dots,b_{i-1},\bar{i},b_{i+1},\dots,b_n \}|}\chi(b_1,\dots,b_{i-1},\bar{i},b_{i+1},\dots,b_n) \\
&= -{_{-A}\chi(b_1,\dots,b_{i-1},\bar{i},b_{i+1},\dots,b_n)} 
\end{align*}
for all $(b_1,\dots,b_n) \in \{ 1,\bar{1}\} \times \dots \times \{ n, \bar{n}\}$ and $i \in [n]$.
Hence ${\cal N}=([2n+1],\{ {_{-A}\chi}, -{_{-A}\chi}\} )$ is a P-matroid.
\end{proof}

\begin{prop} \label{prop:IsoFacPUSO}
Let $\cal M$ be a P-matroid and ${\widehat {\cal M}}$ an extension that defines a nondegenerate POMCP.
\begin{itemize}
\item[(1)]
For every C-equivalent $\cal N$ to $\cal M$, there exists an extension inducing an orientation that is isomorphic to the one induced by ${\widehat {\cal M}}$.
\item[(2)]
For every FS-equivalent $\cal N$ to $\cal M$, there exists an extension inducing an orientation that can be transformed into to the one induced by ${\widehat {\cal M}}$ by successive facet switches.
\end{itemize}
\end{prop}
\begin{proof}
 Suppose that $\widehat{\cal M}=([2n+1],\{ \widehat{\chi}, -\widehat{\chi}\})$ is of rank $n$.

(1). Let $\sigma$ be a permutation on $[2n]$ with $\sigma (\overline{i}) = \overline{\sigma (i)}$ for all $i \in [n]$ such that ${\cal N}= \sigma \cdot {\cal M}$. The oriented matroid $\widehat{\sigma} \cdot \widehat{\cal M}$, where $\widehat{\sigma}$ is the permutation on $[2n+1]$ with $\widehat{\sigma}|_{[2n]}=\sigma$, is an extension of P-matroid ${\cal N}$. The orientation of the edge connecting $B \subseteq [n]$ with $B \oplus \{i\}$ in the USO induced by $\widehat{\sigma} \cdot \widehat{\cal M}$ is determined by 
\begin{equation*}
 \begin{split}
 {\widehat\chi}(\sigma (b_1), \dots, \sigma (b_n)) & \cdot {\widehat \chi}(\sigma (b_1),\dots, \sigma (b_{i-1}) ,2n+1, \sigma (b_{i+1}),\dots, \sigma (b_{n})),
 \end{split}
\end{equation*}
where $b_j:=j+n$ if $j \in B$ and $b_j:=j$ otherwise for all $j \in [n]$. By (C2) in Definition \ref{def:ChiroAxioms}, the orientation therefore corresponds to the orientation of the edge in the USO induced by ${\widehat {\cal M}}$ that connects $C$ with $C \oplus \{\sigma (b_i) \mod n \}$, where 
$C:=\{ \sigma (b_j) - n \mid \sigma (b_j) > n, j \in [n] \}.$ 
In other words, the permutation $\sigma$ induces
 an isomorphism between the oriented $n$-cubes.

(2). Let $A \subseteq [2n]$ with $\overline{A}=A$ be such that ${\cal N}= {_{-A}{\cal M}}$. Then $_{-A} \widehat{\cal M}$ is an extension of P-matroid ${\cal N}$. The orientation of the edge connecting $B \subseteq [2n]$ with $B \oplus \{i\}$ in the USO induced by $_{-A} \widehat{\cal M}$ is determined by 
\begin{equation*}
 \begin{split}
 {_{-A}{\widehat\chi}}(b_1,\dots,b_n) & \cdot {_{-A}{\widehat \chi}}(b_1,\dots,b_{i-1},2n+1,b_{i+1},\dots,b_{n}) \\
   & = (-1)^{|A \cap \{i\}|} {\widehat\chi}(b_1,\dots,b_n) \cdot {\widehat \chi}(b_1,\dots,b_{i-1},2n+1,b_{i+1},\dots,b_{n}),
 \end{split}
\end{equation*}
where $b_j=j+n$ if $j \in B$ and $b_j=j$ otherwise for all $j \in [n]$. Hence, the orientation is opposite to the orientation of the same edge in the USO induced by ${\widehat {\cal M}}$ if and only if $i \in A$. 
\end{proof}\\
\\
Let $\Pi$ be the set of CFS-equivalence classes of P-matroids in OM($n,2n$).
For each $\pi \in \Pi$, we store the representative P-matroid ${\cal M}_{\pi}$, which is determined by lexicographic comparison
of the reverse lexicographic expressions of chirotopes.

\subsubsection*{{\bf Step 3.}}
The next step is to enumerate all uniform extensions of the CFS-equivalence class representatives ${\cal M}_{\pi}$ obtained in Step 2.
We use the algorithm by Finschi and Fukuda~\cite{FF02,FF03} with a slight modification. In order to reduce the workload, we make use of 
the following relation between the extensions of reorientation equivalent oriented matroids.
\begin{prop}
\label{prop:perm_reori}
Let ${\cal M}$ and ${\cal N}$ be rank $r$ oriented matroids on $E$ with ${\cal N} = {_{-A}(\sigma \cdot {\cal M})}$ for any permutation $\sigma$ on $E$ and subset $A \subseteq E$. There is a one-to-one correspondence between the extensions of ${\cal M}$ and those of ${\cal N}$:
\[ {\widehat {\cal M}} \leftrightarrow {_{-A}({\widehat \sigma} \cdot {\widehat {\cal M}})}, \]
where
${\widehat {\cal M}}$ is an extension of ${\cal M}$ on the ground set $E \cup \{ p \}$ and 
${\widehat \sigma}$ is the permutation on $E \cup \{ p \}$ with ${\widehat \sigma}|_E = \sigma$ and ${\widehat \sigma}(p)=p$.
\end{prop}
\begin{proof}
Pick an arbitrary extension ${\widehat {\cal M}}$ of ${\cal M}$.
Since the oriented matroid ${\widehat {\cal N}}:={_{-A}({\widehat \sigma} \cdot {\widehat {\cal M}})}$ is such that
\[ {\widehat {\cal N}}|_{E} = {_{-A}(\sigma \cdot {\cal M})} = {\cal N},\]
it is an extension of ${\cal N}$. Conversely, for any extension ${\widehat {\cal N}}$ of ${\cal N}$, oriented matroid ${\widehat {\cal M}} := {\widehat \sigma}^{-1} \cdot ( _{-A}{\widehat {\cal N}})$ is an extension of ${\cal M}$.
\end{proof}
\\
\\
We take the following approach to enumerate all uniform extensions of the CFS-equivalence class representatives 
obtained in Step 2.
First, we partition $\Pi$ with respect to reorientation equivalence. 
For each reorientation class $\theta$, we store a representative ${\cal M}_{\theta}$ along with information to recover the 
CFS-equivalence class representatives ${\cal M}_{\pi}$ contained in $\theta$.
For every representative ${\cal M}_{\theta}$, we then compute all uniform extensions using
Finschi and Fukuda's algorithm~\cite{FF02,FF03}.
Finally, the uniform extensions of CFS-equivalence class representatives ${\cal M}_{\pi}$ contained in $\theta$ are obtained using Proposition \ref{prop:perm_reori}.

\subsubsection*{{\bf Step 4.}}
Once having obtained all uniform extensions, it is easy to find the induced POMCP-orientations using~\eqref{OMCP}.
We group the extensions of CFS-equivalence class representatives ${\cal M}_{\pi}$ such that extensions inducing the same orientation build one group.

\subsubsection*{{\bf Step 5.}} 
We do not know which POMCP-orientations actually correspond to PLCP-orientations so far.
A POMCP-orientation ${\cal O}$ is a PLCP-orientation if and only if there is at least one realizable oriented matroid in the group 
of extensions that induce ${\cal O}$. 
Therefore, deciding whether ${\cal O}$ arises from a PLCP is done by checking realizability of oriented matroids. Recall that PLCP-orientations are closed under isomorphism and facet switches~\cite{SW78}. We partition the set of POMCP-orientations accordingly, and remember representative USOs. It then suffices to check for each representative USO whether it is a PLCP-orientation. This is done by looking for a realizable oriented matroid in the group of inducing extensions. We check realizability by using the method in \cite{FMM12}.

\section{Experimental results}
\label{sec:exp}
In the first step, the P-matroids in OM($3,6$) and OM($4,8$) were extracted from the database provided by Finschi and Fukuda~\cite{FF}. Then, we partitioned the set of the P-matroids with respect to CFS-equivalence and computed all extensions. We observed that there exists at least one class of CFS-equivalent P-matroids in every reorientation class of uniform oriented matroids in OM($3,6$) and OM($4,8$).
The results are summarized in Table~\ref{number_pmatroids}.\\

\begin{table} [h]
 \centering
 \begin{tabular}{lcc}
   & OM$(3,6)$ & OM$(4,8)$ \\ [0.5ex] \hline \\ [-2ex]
  P-matroids up to C-equivalence & 19 & 156,691\\
  P-matroids up to CFS-equivalence & 13 & \phantom{1}19,076 \\
  P-matroid extensions & 1,920 &1,334,887,042 
 \end{tabular}
 \caption{The number of uniform P-matroids \label{number_pmatroids}}
\end{table}


Next, we computed the induced POMCP-orientations and built the classes defined by isomorphism and facet switches. See row 1 in Table~\ref{number_orientations}.
For each class of USOs, we randomly picked inducing oriented matroids and checked realizability.
For POMCP-orientations of the $3$-cube (resp.~$4$-cube), we have to investigate realizability of oriented matroids
in OM($3,7$) (resp.~OM($4,9$)). Since every oriented matroid in OM($3,7$) is realizable~\cite{C71,H71}, all $8$ POMCP-orientations of the $3$-cube are PLCP-orientations. Up to isomorphism only, there are $17$ POMCP-orientations of the the $3$-cube. Among them, $16$ are acyclic and $1$ is cyclic. See row 3 and 4 in Table~\ref{number_orientations}.
This result coincides with the result obtained by Stickney and Watson~\cite{SW78}. For the $589$ classes of POMCP-orientations of the $4$-cube, realizability checks are necessary. One may expect that some of these orientations are non-realizable, but to our surprise all of them are realizable. For every orientation, the first pick of an inducing oriented matroid turned out to be realizable, coincidentally.
Hence, all $589$ POMCP-orientations of the $4$-cube are PLCP-orientations. There are $6,910$ orientations up to isomorphism. Among them, $5,951$ are acyclic and $959$ cyclic. 
\begin{table} [h!]
 \centering
 \begin{tabular}{lcc}
   & $3$-cube & $4$-cube \\ [0.5ex] \hline \\ [-2ex]
	POMCP-orientations up to facet switches & \phantom{1}8 & \phantom{6,}589 \\
	PLCP-orientations up to facet switches & \phantom{1}8 & \phantom{6,}589 \\
  POMCP-orientations  & 17 & 6,910 \\
	acyclic POMCP-orientations & 16 & 5,951 
  
 \end{tabular}
 \caption{The number of POMCP-orientations up to isomorphism} \label{number_orientations}
\end{table}

Develin~\cite{D06} presented a scheme to construct $2^{\Omega(2^n/\sqrt{n})}$ USOs of the $n$-cube that satisfy the Holt-Klee property. All 4 (resp.~8) USOs of the $3$-cube (resp.~$4$-cube) in the family are PLCP-orientations.
There are at most $2^{O(n^3)}$ PLCP-orientations of the $n$-cube~\cite{FGKS13}. Thus, some higher dimensional USOs in the family are not PLCP-orientations. It would be interesting to investigate whether each one is a POMCP-orientation.

\subsection{Other insights}
As remarked in the previous subsection, every uniform oriented matroid in OM($4,8)$ has a reorientation that is a P-matroid.
This leads to the following open question.
\begin{itemize}
 \item For any $n \in \mathbb{N}$, does every uniform oriented matroid in OM($n,2n$) has a reorientation that is a P-matroid?
\end{itemize}
In case the answer is ``yes,'' the combinatorial structure of P-matrices can be arbitrary complicated.
The following related questions arise.
\begin{itemize}
\item Is the number of rank $n$ P-matroids of the same order of magnitude
as the number of uniform oriented matroids of rank $n$ on $2n$ elements?
\item In \cite[Section 7.4]{BLSWZ99}, they construct $2^{c_r(n-r)^{r-1}}$, where $c_r=(\frac{1}{2(r-1)})^{r-1}$, uniform oriented matroids of rank $r$ on $n$ elements. 
Does each oriented matroid in this family with $n=2m$ and $r=m$ for $m \in \mathbb{N}$ has a reorientation that is a P-matroid?
\end{itemize}

But then again, many PLCP-orientations can be generated from some specific class of coefficient matrices.
We observed that $6,077$ of the $6,910$ PLCP-orientations of the $4$-cube arise from POMCPs 
where at least one of the underlying P-matroids is a reorientation of the alternating matroid ${\cal A}_{4,8}$. 
The {\it alternating matroid} ${\cal A}_{r,n}$ is the rank $r$ oriented matroid on the ground set $[n]$ whose chirotope $\chi$ satisfies
\[ \chi (i_1,\dots, i_r) = + \text{ for all $1 \leq i_1 < \dots < i_r \leq n$.}\]
Every alternating matroid is realized by a cyclic polytope.

Let a {\it cyclic-P-matroid} be a P-matroid of rank $n$ that is reorientation equivalent to the alternating matroid ${\cal A}_{n,2n}$. Since alternating matroids have a good characterization and their extensions are well-studied~ \cite{Z91}, investigation of the POMCP-orientations arising from cyclic-P-matroids becomes interesting. In \cite{FKM,K12}, it is proven that the class of cyclic-P-matroids is closed under CFS-equivalence and some minor operations. A characterization of the reorientations of alternating matroids ${\cal A}_{n,2n}$ that yield P-matroids is presented as well.

\subsection{Conclusions}
In this paper, we presented the complete enumeration of PLCP-orientations of the $4$-cube through solving the realizability problem of 
oriented matroids.
The related data, including the PLCP-orientations of the 4-cube and the corresponding realizing oriented matroids, is uploaded to 

\vspace{+2mm}

\ \ \ {\tt https://sites.google.com/site/hmiyata1984/lcp}.

\vspace{+2mm}

As mentioned, every POMCP-orientation of the $4$-cube is a PLCP-orientation. It is an open question whether there is a gap between the two classes in higher dimensions.

The computational experiments suggest that the combinatorial structure of P-matrices can be complicated. 
The experiments also point us to the interesting class of cyclic-P-matrix LCPs~\cite{FKM,K12}.

Another possibility to make use of the database might be to improve on various bounds on the runtime complexity of algorithms for the PLCP.
Let $t(n)$ be the smallest (among all strategies) maximum (among all USOs and choices of initial vertices) 
number of vertex evaluations necessary to find the sink of a USO of the $n$-cube. 
Szab\'o and Welzl~\cite{SW01} proposed the \emph{Improved Fibonacci Seesaw} algorithm. 
The runtime analysis yields the recurrence relation $t(n) \leq 2 + \sum_{i=0}^{n-5}{t(i)} + 5t(n-4)$ for $n \geq 5$. 
They provide the initial values $t(0)=1,t(1)=2,t(2)=3,t(3)=5$, and $t(4) \leq 7$ to obtain a bound of $O(1.61^n)$ vertex evaluations. 
The database may be used to find good initial values in similar contexts. In this particular case, the question is whether 
the sink of PLCP-orientations of the $4$-cube can be found by evaluating at most $6$ vertices.

Finally, we remark that enumeration of the LP-orientations of the $4$-cube might also be possible. Recall that for any LP $\min c^Tx$ subject to 
$Ax \leq b$ and $x \geq 0$, the \emph{reduced cost vectors}, which determine the LP-orientation, do not depend on the right-hand side $b$. 
There is hope that the LP-orientations of the $4$-cube can be obtained through considering oriented matroids of rank $4$ on $9$ elements. 
In fact, Foniok, Fukuda, and Klaus~\cite{FFK_pre, K12} established a relation between the {\it hidden K-matroid OMCP} of order $n$ and 
the oriented matroid programming~\cite{FKT82} over $n$-cubes. In the realizable case, the LP-orientations of the $n$-cube coincide with the 
USOs arising from the hidden K-matrix LCPs of order~$n$~\cite{FFK_pre, K12, R07}. Hence, one approach is to first identify the hidden K-matroids
among the P-matroids in OM$(4,8)$, then to create all extensions, and finally to check realizability. Since hidden K-matroids of rank $n$,
unlike P-matroids, admit realizations $(I_n,-M) \in \mathbb{R}^{n \times 2n}$ where $M$ is not a hidden K-matrix, such an approach is 
likely to result in a proper superset of the LP-orientations. Another possibility would be to check for every acyclic PLCP-orientation 
of the $4$-cube whether it is an LP-orientation. Such a problem reduces to the realizability question for rank $5$ oriented matroids on $10$ elements. 
The method in \cite{FMM12} cannot handle instances of that size, which might change with some algorithmic improvements.


\subsection*{Acknowledgments}
The authors would like to thank Komei Fukuda for suggesting to study enumeration of PLCP-orientations of the $4$-cube.
They are also grateful to anonymous referees for many helpful comments and suggestions, 
which substantially improved an earlier draft of this paper.
The second author is partially supported by JSPS Grant-in-Aid for Young Scientists (B) 26730002.

\end{document}